\documentclass[11pt,a4paper,reqno]{amsart}

\usepackage{amsfonts}
\usepackage{amsmath}
\usepackage{amssymb}
\usepackage{amsthm}
\usepackage{mathtools}
\usepackage[hidelinks]{hyperref}
\hypersetup{
	pdftitle={Flow decomposition and stochastic completeness of weighted graphs},
	pdfauthor={Qingsong Gu, Lu Hao, Xueping Huang, Yuhua Sun},
	pdfsubject={Stochastic completeness of weighted graphs via flow decomposition},
	pdfkeywords={flow decomposition, stochastic completeness, intrinsic metric, capacity to infinity, relative capacity}
}
\usepackage{microtype}

\newcommand{\Div}{\operatorname{div}}
\newcommand{\capacity}{\operatorname{Cap}}
\newcommand{\cE}{\mathcal E}
\newcommand{\cD}{\mathcal D}
\newcommand{\cT}{\mathcal T}
\newcommand{\tail}{\operatorname{tail}}
\newcommand{\head}{\operatorname{head}}
\newcommand{\1}{\mathbf 1}
\newcommand{\dd}{\mathrm d}
\newcommand{\E}{\mathbb E}
\newcommand{\Pp}{\mathbb P}

\numberwithin{equation}{section}
\newtheorem{theorem}{Theorem}[section]
\newtheorem{lemma}[theorem]{Lemma}
\newtheorem{proposition}[theorem]{Proposition}
\newtheorem{corollary}[theorem]{Corollary}
\newtheorem{remark}[theorem]{Remark}

\newif\ifprintstandardproofs
\printstandardproofsfalse

\newif\ifprintnashwilliams
\printnashwilliamsfalse

\allowdisplaybreaks

\begin{document}
	
	\author{Qingsong Gu}
	\address{School of Mathematics, Nanjing University, Nanjing 210093, P. R. China}
	\email{qingsonggu@nju.edu.cn}
	
	\author{Lu Hao}
	\address{Universit\"{a}t Bielefeld, Fakult\"{a}t f\"{u}r Mathematik, Postfach 100131, D-33501 Bielefeld, Germany}
	\email{lhao@math.uni-bielefeld.de}
	
	\author{Xueping Huang}
	\address{School of Mathematics and Statistics, Nanjing University of Information Science and Technology, Nanjing 210044, P. R. China}
	\email{hxp@nuist.edu.cn}
	
	\author{Yuhua Sun}
	\address{School of Mathematical Sciences and LPMC, Nankai University, Tianjin 300071, P. R. China}
	\email{sunyuhua@nankai.edu.cn}
	
	\title[Stochastic completeness]{Flow decomposition and stochastic completeness of weighted graphs}
	
	\thanks{Q. Gu was supported by the National Natural Science Foundation of China (Grant Nos. 12101303 and 12171354). L. Hao was funded by the Deutsche Forschungsgemeinschaft (DFG, German Research Foundation), Project-ID 317210226, SFB 1283. X. Huang was supported by the National Natural Science Foundation of China (Grant No. 11601238). Y. Sun was funded by the National Natural Science Foundation of China (Grant No. 12371206) and the Fundamental Research Funds for the Central Universities (No. 050-63263078).}
	
	
	\begin{abstract}
		We give a new proof by flow decomposition of the known Grigor'yan type
		integral criterion for stochastic completeness of weighted graphs.  The main
		ingredient is a finite-domain estimate relating the mass of a killed
		$1$-resolvent and the current reaching the boundary to relative capacities
		of intrinsic balls.  This estimate yields logarithmic criteria in terms of
		capacity to infinity and relative capacity, and recovers the volume criterion
		by an intrinsic cutoff, without passing to a metric graph or refining the
		graph.
	\end{abstract}
	
	\maketitle
	\tableofcontents
	\newpage
	
	\section{Introduction}\label{sec:introduction}
	
	A weighted graph carries a minimal reversible continuous-time Markov chain.
	The process may make infinitely many jumps in finite time; stochastic
	completeness rules out this phenomenon.  Equivalently,
	\begin{equation*}
		P_t\1=\1,
		\qquad\text{or}\qquad
		\Pp_x(\zeta=\infty)=1\quad\text{for every }x\in V,
	\end{equation*}
	where $(P_t)_{t\ge0}$ is the heat semigroup and $\zeta$ is the lifetime of the
	minimal process.
	
	For Brownian motion on a complete Riemannian manifold, Grigor'yan proved that
	stochastic completeness follows from
	\begin{equation}\label{eq:grigoryan-manifold-intro}
		\int^\infty
		\frac{r\,\dd r}{\log \mu(B(o,r))}
		=\infty.
	\end{equation}
	This is a basic example of how large-scale volume growth controls the long-time
	behaviour of a diffusion \cite{Grigoryan1987,Grigoryan1999}.
	
	The graph analogue has been studied from several points of view.  Weber and
	Wojciechowski initiated the systematic study of stochastic completeness for
	weighted graphs \cite{Weber2010,Wojciechowski2011}.  Keller and Lenz placed
	the subject in a general Dirichlet-form framework \cite{KellerLenz2012}.
	Intrinsic metrics for nonlocal Dirichlet forms were developed in
	\cite{FrankLenzWingert2014}, while related adapted metrics were used in a
	uniqueness class for the discrete heat equation
	\cite{Huang2012Uniqueness}.
	
	The nonlocality of the graph Laplacian makes the sharp counterpart of
	\eqref{eq:grigoryan-manifold-intro} more delicate than in the manifold
	setting.  Folz first proved the sharp integral criterion for intrinsic 
	metrics by passing to an associated metric graph \cite{Folz2014}.  Huang
	gave alternative analytic proofs \cite{Huang2014Volume}.
	Huang, Keller, and Schmidt obtained the optimal formulation for intrinsic
	pseudometrics with finite balls by combining a heat-equation uniqueness class
	with graph refinements \cite{HuangKellerSchmidt2020}.  Further background can
	be found in \cite{Huang2025Survey,KellerLenzWojciechowski2021}.
	
	In this paper, we give a new proof by flow decomposition of the known
	Grigor'yan type integral criterion for stochastic completeness of weighted
	graphs.  The strongest capacitary result obtained here is a logarithmic
	criterion in terms of capacity to infinity.  It subsumes a
	relative-capacity criterion, from which the volume criterion follows by an
	intrinsic cutoff.  The entire argument is carried out on the original
	weighted graph; no metric graph or graph refinement is introduced.
	
	Flow decomposition and the basic relative-capacity estimate were used in our
	recent work on semilinear Lane--Emden inequalities
	\cite{GuHaoHuangSunLaneEmden2026}.  The linear problem considered here
	requires a different pathwise logarithmic estimate.  It also requires us to
	account for current absorbed by the killing term before it reaches the
	boundary.
	
	We now state the main results.  Let $(V,b,m)$ be an infinite, connected,
	locally finite weighted graph.  Thus $b(x,y)=b(y,x)\ge0$, $b(x,x)=0$, and
	$m(x)>0$.  We use the Laplacian
	\begin{equation*}
		\Delta f(x)
		=\frac1{m(x)}\sum_{y\in V}b(x,y)(f(y)-f(x)).
	\end{equation*}
	We write $x\sim y$ when $b(x,y)>0$.  An edge-length function $\rho$ assigns
	a positive length $\rho(x,y)=\rho(y,x)$ to every edge $x\sim y$.  We call
	$\rho$ $m$-adapted if
	\begin{equation*}
		\sum_{y\in V}b(x,y)\rho(x,y)^2\le m(x),
		\qquad x\in V.
	\end{equation*}
	The induced path metric $d_\rho$ is called an intrinsic path metric.  Fix
	$o\in V$ and write
	\begin{equation*}
		B_r:=B_\rho(o,r)=\{x\in V:d_\rho(o,x)\le r\}.
	\end{equation*}
	Throughout the paper, we assume that $d_\rho$ is complete; equivalently, all
	its closed balls are finite
	\cite[Theorem~11.16]{KellerLenzWojciechowski2021}.
	
	For a finite set $\Omega\subset V$ and $A\subset\Omega$, define the relative
	capacity
	\begin{equation*}
		\capacity_\Omega(A)
		:=\inf\bigl\{\cE(\varphi):
		\varphi=1\text{ on }A,\quad
		\varphi=0\text{ on }V\setminus\Omega\bigr\},
	\end{equation*}
	where
	\begin{equation*}
		\cE(\varphi)
		:=\frac12\sum_{x,y\in V}b(x,y)(\varphi(x)-\varphi(y))^2.
	\end{equation*}
	For a finite set $A\subset V$, its capacity to infinity is
	\begin{equation*}
		\capacity_V(A)
		:=\inf\{\cE(\varphi):
		\varphi\text{ is finitely supported and }\varphi\ge1\text{ on }A\}.
	\end{equation*}
	
	The capacity-to-infinity criterion is as follows.
	
	\begin{theorem}\label{thm:intro-capacity-infinity}
		Let $d_\rho$ be a complete intrinsic path metric associated with an
		$m$-adapted edge-length function $\rho$.  If, for some $o\in V$,
		\begin{equation}\label{eq:intro-capacity-infinity-condition}
			\int_1^\infty
			\frac{r\,\dd r}
			{\log\!\left(e+\capacity_V(B_r)/m(o)\right)}
			=\infty,
		\end{equation}
		then $(V,b,m)$ is stochastically complete.
	\end{theorem}
	
	\begin{remark}\label{rem:recurrence-capacity}
		For the weighted graph $(V,b,m)$, recurrence is equivalent to the
		vanishing of the effective conductance from any vertex to infinity
		\cite[Theorem~2.3, p.~27]{LyonsPeres2016}.  By the Dirichlet principle, the
		effective conductance from a finite set $A$ to infinity is $\capacity_V(A)$;
		see \cite[Exercise~2.93, p.~68]{LyonsPeres2016}.  Since
		\[
		\capacity_V(A)
		\le\sum_{x\in A}\capacity_V(\{x\}),
		\]
		recurrence is therefore equivalent to $\capacity_V(A)=0$ for every nonempty
		finite set $A$.  Thus, on a recurrent weighted graph, $\capacity_V(B_r)=0$ for every
		$r$, and
		\eqref{eq:intro-capacity-infinity-condition} holds automatically.
		Within the standing complete-intrinsic-metric setting,
		Theorem~\ref{thm:intro-capacity-infinity} therefore recovers the  fact that
		\[
		\text{recurrence}\quad\Longrightarrow\quad\text{stochastic completeness}.
		\]
		The implication itself is true in general, without the completeness assumption. 
	\end{remark}
	
	Since
	\begin{equation}\label{eq:intro-global-local-capacity}
		\capacity_V(B_r)\le \capacity_{B_{2r}}(B_r),
	\end{equation}
	Theorem~\ref{thm:intro-capacity-infinity} subsumes the following
	relative-capacity criterion.
	
	\begin{theorem}\label{thm:intro-capacity}
		Let $d_\rho$ be a complete intrinsic path metric associated with an
		$m$-adapted edge-length function $\rho$.  If, for some $o\in V$,
		\begin{equation}\label{eq:intro-capacity-condition}
			\int_1^\infty
			\frac{r\,\dd r}
			{\log\!\left(e+\capacity_{B_{2r}}(B_r)/m(o)\right)}
			=\infty,
		\end{equation}
		then $(V,b,m)$ is stochastically complete.
	\end{theorem}
	
	The standard intrinsic cutoff gives
	\begin{equation}\label{eq:intro-cutoff}
		\capacity_{B_{2r}}(B_r)
		\le \frac{m(B_{2r})}{r^2}.
	\end{equation}
	Consequently, Theorem~\ref{thm:intro-capacity} yields the following known
	Grigor'yan type integral criterion.  Folz first proved this graph result under
	intrinsic path-metric assumptions \cite{Folz2014}; see
	\cite{Huang2014Volume,HuangKellerSchmidt2020} for the later formulations and
	alternative proofs described above.
	
	\begin{corollary}\label{cor:intro-volume}
		Under the same assumptions as Theorem~\ref{thm:intro-capacity}, suppose that,
		for some $o\in V$,
		\begin{equation}\label{eq:intro-volume-condition}
			\int_1^\infty
			\frac{r\,\dd r}
			{\log\!\left(e+m(B_r)/m(o)\right)}
			=\infty,
		\end{equation}
		then $(V,b,m)$ is stochastically complete.
	\end{corollary}
	
	The proof begins with the killed $1$-resolvent on a finite domain.  Its
	gradient defines a current with one unit of source at the pole; part of this
	current is absorbed by the killing term, and the remainder reaches the
	boundary.  Flow decomposition and the pathwise logarithmic estimate lead to
	the central finite-domain estimate in
	Lemma~\ref{lem:finite-domain-capacity}.  Applying this lemma to intrinsic balls
	and passing to an exhaustion proves
	Theorem~\ref{thm:intro-capacity-infinity}.  The relative-capacity criterion
	follows from \eqref{eq:intro-global-local-capacity}, and the intrinsic cutoff
	\eqref{eq:intro-cutoff} gives Corollary~\ref{cor:intro-volume}.
	
	Relative capacity also appears in other criteria for stochastic completeness,
	but in different roles.  Andres and Barlow combine cutoff Sobolev inequalities
	with volumes of annuli \cite{AndresBarlow2015}.  Grigor'yan and Masamune use
	relative capacity in connection with the Cauchy boundary and Green's formula
	\cite{GrigoryanMasamune2013}.  We have not found in the literature the
	logarithmic capacity criteria stated above or the finite-domain estimates
	proved in Section~\ref{sec:flow-capacity}.  The capacity-to-infinity form is
	obtained by exhausting the finite-domain estimate.
	
	For weighted half-lines, stochastic completeness is characterized by the
	divergence of the exact birth--death sum
	\begin{equation*}
		\sum_{n=0}^\infty\frac{M_n}{b_n},
	\end{equation*}
	where $M_n$ is the mass of $\{0,\ldots,n\}$ and $b_n=b(n,n+1)$; see
	\cite[Theorem~9.25, pp.~406--408]{KellerLenzWojciechowski2021}.  We recover
	the corresponding finite identity directly from the killed $1$-resolvent.
	Two weighted-half-line examples then clarify the relation between the general
	sufficient conditions.
	The first shows that the relative-capacity condition can hold while the volume
	condition fails.  In fact, the capacity to infinity vanishes in that example,
	so the capacity-to-infinity condition holds as well.  The second shows that all
	three general criteria can still lose a logarithmic order relative to the exact
	birth--death test.
	
	The paper is organized as follows.  Section~\ref{sec:preliminaries} recalls
	weighted graphs, intrinsic path metrics, finite killed resolvents, flows, and
	capacities.  In Section~\ref{sec:flow-capacity} we prove the
	finite-domain capacity estimate by flow decomposition and derive its
	relative-capacity form.  Section~\ref{sec:killed} applies the estimate to
	killed $1$-resolvents and proves the capacity-to-infinity, relative-capacity,
	and volume criteria.  Section~\ref{sec:half-line} contains the exact
	weighted-half-line test and two examples.

	\section{Preliminaries}\label{sec:preliminaries}
	
	In this section, we fix the conventions used below and recall the standard
	finite-network facts needed later.  For general background on flows,
	resolvents, and weighted graphs, see
	\cite{FordFulkerson1962,LyonsPeres2016,KellerLenzWojciechowski2021}.
	No originality is claimed in this section.
	
	\subsection{Weighted graphs and the minimal process}
	
	Let $(V,b,m)$ be as in the introduction.  We write $x\sim y$ when
	$b(x,y)>0$.  Local finiteness means that every vertex has finitely many
	neighbors.  The jump rate from $x$ to $y$ is $b(x,y)/m(x)$, and the total
	jump rate at $x$ is finite.  The resulting minimal process is denoted by
	$(X_t)_{t\ge0}$, its lifetime by $\zeta$, and its heat semigroup by $P_t$.
	We send the process to a cemetery state after $\zeta$.  Then
	\begin{equation*}
		P_t\1(x)=\Pp_x(t<\zeta).
	\end{equation*}
	For the construction of the jump process and this identity, see
	\cite[Section~2.5 and Theorem~7.32]{KellerLenzWojciechowski2021}.
	
	For a finite set $D\subset V$, let
	\begin{equation*}
		\tau_D:=\inf\{t\ge0:X_t\notin D\}.
	\end{equation*}
	The killed semigroup is
	\begin{equation*}
		P_t^D f(x)=\E_x\bigl[f(X_t)\1_{\{t<\tau_D\}}\bigr],
		\qquad x\in D.
	\end{equation*}
	
	By the discrete Hopf--Rinow theorem, a path metric induced by positive edge
	lengths on a connected locally finite graph is complete if and only if every
	closed ball is finite; see
	\cite[Theorem~11.16]{KellerLenzWojciechowski2021}.  We also use the
	following elementary consequence.
	
	\begin{lemma}\label{lem:balls-connected}
		Let $d$ be a complete path metric induced by positive edge lengths on a
		connected locally finite graph.  Then every closed $d$-ball is connected in
		the underlying graph.
	\end{lemma}
	
	\begin{proof}
		Fix $x\in B_d(o,R)$.  By completeness and the discrete Hopf--Rinow
		theorem, $B_d(o,R+1)$ is finite.  Choose a path from $o$ to $x$ of length
		less than $d(o,x)+1$, and erase its loops.  The resulting simple path has
		length less than $R+1$, and every vertex on it belongs to $B_d(o,R+1)$.
		
		There are only finitely many simple paths from $o$ to $x$ contained in
		$B_d(o,R+1)$.  Choose one of minimal length and call it $\gamma$.  Then
		$\gamma$ is minimizing among all paths from $o$ to $x$.  Indeed, any shorter
		path could be made simple by erasing loops; since its total length would be
		less than $R+1$, it would also lie in $B_d(o,R+1)$, contradicting the choice
		of $\gamma$.  Thus $\gamma$ has length $d(o,x)$.
		
		If $z$ is a vertex of $\gamma$, then
		\[
		d(o,z)
		\le \operatorname{length}(\gamma|_{o\to z})
		\le \operatorname{length}(\gamma)
		=d(o,x)
		\le R.
		\]
		Hence $\gamma\subset B_d(o,R)$, and the ball is connected.
	\end{proof}
	
	\subsection{Energy and capacities on the ambient graph}
	
	For a finitely supported function $f:V\to\mathbb R$, define
	\begin{equation*}
		\cE(f)
		:=\frac12\sum_{x,y\in V}b(x,y)(f(x)-f(y))^2.
	\end{equation*}
	If $\Omega\subset V$ is finite and $A\subset\Omega$, the capacity of $A$
	relative to $\Omega$ is
	\begin{equation}\label{eq:relative-capacity-domain}
		\capacity_\Omega(A)
		:=\inf\{\cE(f):f=1\text{ on }A,\ f=0\text{ on }V\setminus\Omega\}.
	\end{equation}
	Every admissible function $f$ is supported in $\Omega$, so only the edges having
	at least one endpoint in $\Omega$ contribute to its energy.
	Moreover, relative capacity is decreasing in the outer domain: if
	$A\subseteq\Omega_1\subseteq\Omega_2$, then
	\begin{equation*}
		\capacity_{\Omega_2}(A)\le \capacity_{\Omega_1}(A).
	\end{equation*}
	
	For a finite set $A\subset V$, define its capacity to infinity by
	\begin{equation}\label{eq:capacity-infinity-definition}
		\capacity_V(A)
		:=\inf\{\cE(f):
		f\text{ is finitely supported and }f\ge1\text{ on }A\}.
	\end{equation}
	Truncating $f$ to $(0\vee f)\wedge1$ does not increase its energy.  Thus one
	may equivalently require $f=1$ on $A$ and $0\le f\le1$.  In electrical
	terminology, $\capacity_V(A)$ is the effective conductance from $A$ to
	infinity; it may be zero.
	
	\begin{lemma}\label{lem:capacity-exhaustion}
		Let $A\subset V$ be finite, and let $(\Omega_j)$ be an increasing exhaustion
		of $V$ by finite sets containing $A$.  Then
		\begin{equation}\label{eq:capacity-exhaustion}
			\capacity_{\Omega_j}(A)\downarrow\capacity_V(A).
		\end{equation}
		In particular, for fixed $r>0$,
		\begin{equation*}
			\capacity_{B_R}(B_r)\downarrow\capacity_V(B_r)
			\qquad\text{as }R\to\infty.
		\end{equation*}
	\end{lemma}
	
	\begin{proof}
		Domain monotonicity shows that the relative capacities decrease, and every
		relative admissible function is admissible in
		\eqref{eq:capacity-infinity-definition}.  Hence their limit is at least
		$\capacity_V(A)$.  Conversely, for $\varepsilon>0$, choose a finitely
		supported $f$, equal to one on $A$, such that
		\[
		\cE(f)\le\capacity_V(A)+\varepsilon.
		\]
		For all sufficiently large $j$, the support of $f$ lies in $\Omega_j$, so
		\[
		\capacity_{\Omega_j}(A)
		\le\cE(f)
		\le\capacity_V(A)+\varepsilon.
		\]
		Letting $j\to\infty$ and then $\varepsilon\downarrow0$ proves the claim.
		The assertion for balls follows because completeness makes them finite and
		they exhaust $V$.
	\end{proof}
	
	\subsection{Finite networks and flows}
	
	Let $(W,c)$ be a finite connected weighted network: $W$ is a finite vertex
	set and $c:W\times W\to[0,\infty)$ is symmetric, with $c(x,x)=0$.  We write
	$x\sim_W y$ when $c(x,y)>0$, and denote the oriented edge set by
	\begin{equation*}
		\overrightarrow E_W:=\{(x,y)\in W\times W:c(x,y)>0\}.
	\end{equation*}
	For $f:W\to\mathbb R$, set
	\begin{equation*}
		\cE_W(f)
		:=\frac12\sum_{x,y\in W}c(x,y)(f(x)-f(y))^2.
	\end{equation*}
	If $A,K\subset W$ are disjoint and nonempty, define
	\begin{equation*}
		\capacity_W(A,K)
		:=\inf\{\cE_W(f):f=1\text{ on }A,\ f=0\text{ on }K\}.
	\end{equation*}
	
	An antisymmetric edge function is a function
	$I:\overrightarrow E_W\to\mathbb R$ satisfying
	$I(y,x)=-I(x,y)$.  Its divergence is
	\begin{equation*}
		\Div I(x):=\sum_{y\sim_W x}I(x,y).
	\end{equation*}
	For disjoint nonempty sets $A,K\subset W$, such an edge function is called a
	flow from $A$ to $K$ if
	\begin{equation*}
		\Div I=0\quad\text{on }W\setminus(A\cup K),
		\qquad
		\Div I\ge0\quad\text{on }A,
		\qquad
		\Div I\le0\quad\text{on }K.
	\end{equation*}
	By antisymmetry,
	\begin{equation*}
		\sum_{x\in W}\Div I(x)=0.
	\end{equation*}
	Since $\Div I$ vanishes outside $A\cup K$, it follows that
	\begin{equation*}
		\sum_{x\in A}\Div I(x)
		=
		-\sum_{x\in K}\Div I(x)
		\ge0.
	\end{equation*}
	We call this common value the strength of $I$ and denote it by
	\begin{equation*}
		|I|:=\sum_{x\in A}\Div I(x).
	\end{equation*}
	Its energy is
	\begin{equation*}
		\cD_W(I):=\frac12\sum_{(x,y)\in\overrightarrow E_W}
		\frac{I(x,y)^2}{c(x,y)}.
	\end{equation*}
	We omit the subscript when the finite network is clear.  Thomson's principle
	gives
	\begin{equation}\label{eq:thomson-finite-network}
		\cD_W(I)\ge \frac{|I|^2}{\capacity_W(A,K)}.
	\end{equation}
	Indeed, if $h$ is the equilibrium potential, then summation by parts and the
	Cauchy--Schwarz inequality yield
	\begin{align*}
		|I|
		&=\sum_{x\in W}h(x)\Div I(x)\\
		&=\frac12\sum_{(x,y)\in\overrightarrow E_W}
		I(x,y)(h(x)-h(y))\\
		&\le \cD_W(I)^{1/2}\cE_W(h)^{1/2}.
	\end{align*}
	
	The ambient and finite-network capacities are related by an
	energy-preserving identification.  Given a finite connected set
	$\Omega\subset V$, retain every edge with both endpoints in $\Omega$.  For
	each boundary edge $\{x,y\}$, with $x\in\Omega$ and $y\notin\Omega$, replace
	$y$ by a separate terminal vertex $t_{\{x,y\}}$, keeping the same edge
	weight.  Denote the resulting network by $\widehat\Omega$ and its terminal
	set by $\cT_\Omega$.
	
	Extension by zero identifies the admissible functions for
	$\capacity_\Omega(A)$ with those for
	$\capacity_{\widehat\Omega}(A,\cT_\Omega)$ and preserves their energies:
	internal edges are unchanged, while each boundary edge is replaced by a
	copied edge with the same weight and a zero-valued terminal endpoint.
	Therefore
	\begin{equation}\label{eq:ambient-terminal-capacity}
		\capacity_\Omega(A)
		=
		\capacity_{\widehat\Omega}(A,\cT_\Omega).
	\end{equation}
	Below, $\capacity_\Omega(A)$ denotes ambient relative capacity, whereas
	$\capacity_{\widehat D}(A,K)$ denotes capacity between disjoint subsets of a
	terminal-copy network.
	
	\subsection{The finite killed \texorpdfstring{$1$}{1}-resolvent}
	
	Let $D\subset V$ be finite and connected.  Functions on $D$ are always
	extended by zero to $V\setminus D$ when a Dirichlet equation is written.  We
	use the network operator
	\begin{equation*}
		Lf(x)
		:=-m(x)\Delta f(x)
		=\sum_{y\in V}b(x,y)(f(x)-f(y)).
	\end{equation*}
	We also use the killed $1$-resolvent density with pole $o\in D$, denoted by
	\begin{equation*}
		g_{1,D}=g_{1,D}(o,\cdot).
	\end{equation*}
	It is characterized by
	\begin{equation}\label{eq:finite-one-resolvent}
		(-\Delta+1)g_{1,D}=\frac{\1_{\{o\}}}{m(o)}\quad\text{in }D,
		\qquad g_{1,D}=0\quad\text{on }V\setminus D.
	\end{equation}
	The subscript $1$ records the resolvent parameter.
	Equivalently, after multiplication by $m(x)$,
	\begin{equation}\label{eq:network-one-resolvent}
		Lg_{1,D}(x)+m(x)g_{1,D}(x)=\1_{\{o\}}(x),
		\qquad x\in D.
	\end{equation}
	By connectedness of $D$ and the strong minimum principle, $g_{1,D}$ is
	strictly positive on $D$.
	
	\section{Flow decomposition and relative capacity estimates}
	\label{sec:flow-capacity}
	
	The central result of this section is
	Lemma~\ref{lem:finite-domain-capacity}, the finite-domain estimate used in the
	remainder of the paper.  Its proof combines flow decomposition, a pathwise
	logarithmic estimate, and a local capacity bound for suitable path segments.
	Closely related forms of the flow decomposition and the local capacity bound
	appeared in \cite{GuHaoHuangSunLaneEmden2026}.  The new points here are the
	logarithmic estimate and the treatment of current absorbed inside the domain
	by the killing term.
	
	Throughout this section, $\rho$ is the $m$-adapted edge-length function fixed
	above, $d_\rho$ is its intrinsic path metric, and
	\[
	B_r=B_\rho(o,r).
	\]
	
	\subsection{The finite current and its flow decomposition}
	
	Let $D\subset V$ be finite and connected, with $o\in D$, and let
	$\widehat D=D\cup\cT_D$ be its terminal-copy network.  Every retained or
	copied edge $e$ inherits a weight $b_e$ and a length $\rho(e)$ from the
	corresponding edge of the original graph.
	Define the radial label by
	\begin{equation*}
		\mathfrak r(z):=
		\begin{cases}
			d_\rho(o,z),&z\in D,\\
			d_\rho(o,y),&z=t_e,
		\end{cases}
	\end{equation*}
	where $y$ is the exterior endpoint represented by $t_e$.  The terminal copies
	are only a finite bookkeeping device; no edge is subdivided and the ambient
	graph, metric, and Markov process are unchanged.
	
	Let $w:\widehat D\to[0,\infty)$ vanish on $\cT_D$, and define the
	antisymmetric gradient current by
	\begin{equation*}
		i_w(x,y):=b_e(w(x)-w(y))
	\end{equation*}
	on every retained or copied edge $e=\{x,y\}$.  Assume that the function
	\begin{equation}\label{eq:current-with-loss}
		\lambda:=\1_{\{o\}}-\Div i_w
	\end{equation}
	is nonnegative on $\widehat D$.  Summing over $\widehat D$ gives
	$\sum_{z\in\widehat D}\lambda(z)=1$.  Thus one unit of current is injected
	at $o$ and absorbed according to $\lambda$.  The value $\lambda(o)$ is
	allowed to be positive.
	
	Orient every edge with nonzero voltage drop from the larger value of $w$ to
	the smaller one, and discard the zero-current edges.  For an edge
	$e=(\tail(e),\head(e))$ in this directed support, set
	\begin{equation*}
		\theta_e
		:=i_w(\tail(e),\head(e))
		=b_e\bigl(w(\tail(e))-w(\head(e))\bigr)>0.
	\end{equation*}
	All sums over the directed support below count each underlying edge once in
	this orientation.  The directed support is acyclic because $w$ decreases
	strictly along every directed edge.  The following finite flow decomposition
	is standard and is also used in \cite{GuHaoHuangSunLaneEmden2026} in a
	slightly different form.
	
	\begin{lemma}\label{lem:flow-decomposition}
		There is a probability measure $\Pi$ on finite directed paths in $\widehat D$
		beginning at $o$ such that
		\begin{align}
			\Pi\{\gamma:\gamma\text{ uses }e\}&=\theta_e,
			\label{eq:flow-edge-marginal}\\
			\Pi\{\gamma:\gamma\text{ ends at }z\}&=\lambda(z).
			\label{eq:flow-terminal-marginal}
		\end{align}
		The path $(o)$ with no edge carries the mass $\lambda(o)$.
	\end{lemma}
	
	\begin{proof}
		Adjoin a cemetery vertex $\dagger$, which serves as a sink, and, for every
		$z\in\widehat D$, add a directed arc
		$z\to\dagger$ carrying flow $\lambda(z)$.  These arcs carry no edge weight or
		length and will not enter any energy estimate.  Equation
		\eqref{eq:current-with-loss} shows that the augmented flow has divergence one
		at $o$, zero at the other vertices of $\widehat D$, and $-1$ at $\dagger$.
		It is acyclic.  Every nonzero residual flow with these divergence properties
		contains a directed path from $o$ to $\dagger$: conservation prevents a path
		in the positive support from stopping at an intermediate vertex, and
		acyclicity prevents it from returning to a previous vertex.  Repeatedly choose
		such a path and subtract the minimum flow carried by one of its arcs.  Each
		step removes an arc from the positive support, so the finite procedure
		terminates.  It produces weighted directed paths whose arc marginals
		reproduce the augmented unit flow; their total weight is one.  Delete the
		final cemetery arc from each path.  The remaining paths satisfy
		\eqref{eq:flow-edge-marginal} and \eqref{eq:flow-terminal-marginal}.  The
		path consisting only of $o\to\dagger$ becomes the zero-edge path at $o$.
	\end{proof}
	
	For a nontrivial path
	\[
	\gamma=(x_0=o,e_0,x_1,\ldots,e_{N-1},x_N)
	\]
	in the support of $\Pi$, set
	\begin{equation*}
		w_i:=w(x_i),\qquad
		\ell_i:=\rho(e_i),\qquad
		\delta_i:=w_i-w_{i+1}>0.
	\end{equation*}
	Since the path follows the directed support,
	\begin{equation*}
		w_0>w_1>\cdots>w_N\ge0.
	\end{equation*}
	Define
	\begin{equation*}
		H(\gamma):=\sum_{i=0}^{N-1}\frac{\ell_i^2w_i}{\delta_i}.
	\end{equation*}
	For the zero-edge path, set $H(\gamma)=0$.
	
	\begin{lemma}\label{lem:average-path-energy}
		One has
		\begin{equation*}
			\E_\Pi H(\gamma)
			\le \sum_{x\in D}w(x)m(x).
		\end{equation*}
		Consequently, if $\Gamma$ is a family of paths with $\Pi(\Gamma)=a>0$, then
		\begin{equation}\label{eq:conditional-path-energy}
			\sum_{x\in D}w(x)m(x)
			\ge a\E_\Pi\bigl[H(\gamma)\mid\Gamma\bigr].
		\end{equation}
	\end{lemma}
	
	\begin{proof}
		By \eqref{eq:flow-edge-marginal},
		\begin{align*}
			\E_\Pi H
			&=\sum_e\theta_e
			\frac{\rho(e)^2w(\tail(e))}
			{w(\tail(e))-w(\head(e))}\\
			&=\sum_e b_e\rho(e)^2w(\tail(e)).
		\end{align*}
		Every tail belongs to $D$, since all terminal copies have value zero.  For a
		fixed $x\in D$, the directed edges in the sum with tail $x$ correspond to a
		subset of the original edges incident with $x$.  Since $\rho$ is
		$m$-adapted,
		\[
		\E_\Pi H
		\le\sum_{x\in D}w(x)\sum_yb(x,y)\rho(x,y)^2
		\le\sum_{x\in D}w(x)m(x).
		\]
		Finally,
		\begin{equation*}
			\E_\Pi H
			\ge\int_\Gamma H\,\dd\Pi
			=a\E_\Pi[H\mid\Gamma],
		\end{equation*}
		which proves \eqref{eq:conditional-path-energy}.
	\end{proof}
	
	\subsection{Pathwise logarithmic estimate}
	
	We begin with an elementary one-dimensional lemma.
	
	\begin{lemma}\label{lem:pathwise-logarithmic}
		Let $v:[0,L]\to[0,\infty)$ be continuous, strictly decreasing, and affine on
		each interval of a finite partition.  Then
		\begin{equation*}
			\int_0^L\frac{v(s)}{-v'(s)}\,\dd s
			\ge \frac1{4\log2}
			\int_0^L
			\frac{s\,\dd s}{\log\!\left(e+v(0)/v(s)\right)},
		\end{equation*}
		where the right-hand integrand is taken to be zero when $v(s)=0$.
	\end{lemma}
	
	\begin{proof}
		Fix $0<T<L$.  All derivatives below are understood almost everywhere; the
		finitely many break points do not affect the integrals.  Strict decrease and
		nonnegativity give $v>0$ on $[0,T]$.  Put
		\[
		a(s):=-\frac{v'(s)}{v(s)},
		\qquad
		A(t):=\int_0^t a(s)\,\dd s
		=\log\frac{v(0)}{v(t)}.
		\]
		Then $a>0$ almost everywhere.
		For $0<t<T$, the Cauchy--Schwarz inequality on $[t/2,t]$ gives
		\[
		\int_{t/2}^t\frac{\dd s}{a(s)}
		\ge\frac{(t/2)^2}{\int_{t/2}^ta(s)\,\dd s}
		\ge\frac{t^2}{4A(t)}.
		\]
		Consequently,
		\begin{align*}
			\int_0^T\frac{t\,\dd t}{A(t)}
			&\le4\int_0^T\frac1t
			\int_{t/2}^t\frac{\dd s}{a(s)}\,\dd t\\
			&=4\int_0^T\frac1{a(s)}
			\int_s^{\min\{2s,T\}}\frac{\dd t}{t}\,\dd s\\
			&\le4\log2\int_0^T\frac{\dd s}{a(s)}\\
			&=4\log2\int_0^T\frac{v(s)}{-v'(s)}\,\dd s.
		\end{align*}
		Since
		\[
		A(t)=\log\frac{v(0)}{v(t)}
		\le\log\!\left(e+\frac{v(0)}{v(t)}\right),
		\]
		the desired right-hand integrand is no larger than $t/A(t)$.  Letting
		$T\uparrow L$ and using monotone convergence proves the result.  If
		$v(L)>0$, one may take $T=L$ directly.
	\end{proof}
	
	For a nontrivial path $\gamma$, set
	\begin{equation*}
		s_i:=\sum_{j=0}^{i-1}\ell_j,
		\qquad 0\le i\le N,
	\end{equation*}
	and let $v_\gamma$ be affine on each $[s_i,s_{i+1}]$ with
	$v_\gamma(s_i)=w_i$.  Then
	\begin{equation}\label{eq:path-energy-integral}
		H(\gamma)
		\ge\int_0^{s_N}
		\frac{v_\gamma(s)}{-v_\gamma'(s)}\,\dd s.
	\end{equation}
	Indeed, on $[s_i,s_{i+1}]$ one has
	$v_\gamma(s_i+t)=w_i-\delta_i t/\ell_i$ and
	$-v_\gamma'=\delta_i/\ell_i$.  Hence
	\begin{align*}
		\int_{s_i}^{s_{i+1}}
		\frac{v_\gamma(s)}{-v_\gamma'(s)}\,\dd s
		&=\frac{\ell_i}{\delta_i}
		\int_0^{\ell_i}\left(w_i-\frac{\delta_i t}{\ell_i}\right)\dd t\\
		&=\frac{\ell_i^2(w_i+w_{i+1})}{2\delta_i}
		\le\frac{\ell_i^2w_i}{\delta_i}.
	\end{align*}
	Summing over $i$ proves \eqref{eq:path-energy-integral}.
	
	The path in $\widehat D$ corresponds to a path in the original graph, with a
	terminal copy replaced by the exterior endpoint that it represents.  Hence
	\begin{equation}\label{eq:prefix-radial}
		\mathfrak r(x_i)\le s_i,
		\qquad 0\le i\le N.
	\end{equation}
	Let $R>0$ satisfy
	\begin{equation*}
		\max_{0\le i\le N}\mathfrak r(x_i)>R.
	\end{equation*}
	For $0<r<R$, define
	\begin{equation*}
		\kappa_r:=\min\{0\le i\le N-1:\mathfrak r(x_{i+1})>r\},
		\qquad Z_r(\gamma):=w_{\kappa_r}.
	\end{equation*}
	The index is well-defined, and its minimality gives
	\begin{equation}\label{eq:first-exit-location}
		\mathfrak r(x_{\kappa_r})\le r
		<\mathfrak r(x_{\kappa_r+1}).
	\end{equation}
	Moreover, $Z_r(\gamma)>0$, since $x_{\kappa_r}$ is the tail of a directed
	edge.
	
	\begin{lemma}\label{lem:discrete-first-exit}
		There is an absolute constant $c>0$ such that, for every nontrivial path in
		the decomposition and every $R$ with
		\begin{equation*}
			0<R<\max_{0\le i\le N}\mathfrak r(x_i),
		\end{equation*}
		one has
		\begin{equation*}
			H(\gamma)
			\ge c\int_0^R
			\frac{r\,\dd r}
			{\log\!\left(e+w(o)/Z_r(\gamma)\right)}.
		\end{equation*}
	\end{lemma}
	
	\begin{proof}
		By Lemma~\ref{lem:pathwise-logarithmic} and
		\eqref{eq:path-energy-integral},
		\begin{equation}\label{eq:path-logarithmic-control}
			\int_0^{s_N}
			\frac{s\,\dd s}{\log\!\left(e+w(o)/v_\gamma(s)\right)}
			\le4\log2\,H(\gamma).
		\end{equation}
		At the first crossing of radius $r$, either the path has already accumulated
		at least $r/2$ units of length, or the crossing edge itself has length
		greater than $r/2$.  Accordingly, split $(0,R)$ into
		\[
		E_1:=\{r:s_{\kappa_r}\ge r/2\},
		\qquad
		E_2:=\{r:s_{\kappa_r}<r/2\}.
		\]
		For $r\in E_1$, monotonicity of $v_\gamma$ gives
		$Z_r\le v_\gamma(r/2)$.  Moreover, $R<s_N$ by
		\eqref{eq:prefix-radial}.  Hence
		\begin{equation}\label{eq:short-edge-radii}
			\begin{aligned}
				\int_{E_1}\frac{r\,\dd r}{\log(e+w(o)/Z_r)}
				&\le\int_0^R
				\frac{r\,\dd r}{\log(e+w(o)/v_\gamma(r/2))}\\
				&=4\int_0^{R/2}
				\frac{s\,\dd s}{\log(e+w(o)/v_\gamma(s))}\\
				&\le16\log2\,H(\gamma).
			\end{aligned}
		\end{equation}
		
		If $r\in E_2$ and $\kappa_r=i$, then
		\eqref{eq:prefix-radial} and \eqref{eq:first-exit-location} imply
		\[
		r<\mathfrak r(x_{i+1})\le s_i+\ell_i<r/2+\ell_i,
		\]
		so $r<2\ell_i$.  Since $w(o)\ge w_i>0$ and
		$0<\delta_i\le w_i$, one has
		$\log(e+w(o)/w_i)>1$, and therefore
		\begin{equation}\label{eq:long-edge-radii}
			\int_{\{r\in E_2:\,\kappa_r=i\}}
			\frac{r\,\dd r}{\log(e+w(o)/w_i)}
			\le2\ell_i^2
			\le2\frac{\ell_i^2w_i}{\delta_i}.
		\end{equation}
		Summing \eqref{eq:long-edge-radii} and using
		\eqref{eq:short-edge-radii} gives
		\[
		\int_0^R
		\frac{r\,\dd r}{\log(e+w(o)/Z_r)}
		\le(2+16\log2)H(\gamma),
		\]
		which proves the lemma.
	\end{proof}
	
	The decomposition into the two sets of radii is needed only for long edges.
	The radii crossed by such an edge are estimated directly by its contribution
	to the path energy, without using a cable construction or an edge refinement.
	
	\subsection{The finite-domain capacity estimate}
	
	We first record the capacity estimate for selected path segments used below.
	A closely related form appeared in
	\cite{GuHaoHuangSunLaneEmden2026}; we include the proof for completeness.  Let
	$\Gamma$ be a family of paths in the flow decomposition with
	$\Pi(\Gamma)=a>0$.  On
	each $\gamma\in\Gamma$, choose a nonempty directed segment starting at a
	vertex $x_{\tau(\gamma)}\in A$ and ending at a vertex
	$x_{\sigma(\gamma)}\in K$, where $A,K\subseteq\widehat D$ are disjoint.  Set
	\begin{equation*}
		Z(\gamma):=w(x_{\tau(\gamma)})>0.
	\end{equation*}
	
	\begin{lemma}\label{lem:selected-segment-capacity}
		With relative capacity computed in the finite graph $\widehat D$,
		\begin{equation}\label{eq:selected-segment-capacity}
			\E_\Pi\left[\frac1Z\,\middle|\,\Gamma\right]
			\le \frac{\capacity_{\widehat D}(A,K)}a.
		\end{equation}
	\end{lemma}
	\begin{proof}
		Use the orientation fixed above, and write $\gamma[\tau,\sigma]$ for the
		selected segment of $\gamma$.  For an edge $e$ in the directed support, define
		\begin{equation*}
			\Phi(e):=
			\int_\Gamma
			\frac{\1_{\{e\in\gamma[\tau,\sigma]\}}}{Z(\gamma)}\,\dd\Pi(\gamma),
		\end{equation*}
		and extend $\Phi$ antisymmetrically, setting it equal to zero off the directed
		support.  The unit edge indicator of a selected segment has divergence $1$ at
		its initial vertex, $-1$ at its terminal vertex, and zero elsewhere.  Hence
		\begin{equation*}
			\Div\Phi(x)=
			\int_\Gamma
			\frac{
				\1_{\{x=x_{\tau(\gamma)}\}}
				-\1_{\{x=x_{\sigma(\gamma)}\}}}
			{Z(\gamma)}\,\dd\Pi(\gamma).
		\end{equation*}
		Thus $\Phi$ is a flow from $A$ to $K$, with strength
		\begin{equation}\label{eq:selected-segment-strength}
			|\Phi|
			=\int_\Gamma\frac{\dd\Pi(\gamma)}{Z(\gamma)}
			=a\E_\Pi\left[Z^{-1}\,\middle|\,\Gamma\right].
		\end{equation}
		
		For every edge $e$ in the directed support, the Cauchy--Schwarz inequality and
		\eqref{eq:flow-edge-marginal} give
		\begin{align*}
			\Phi(e)^2
			&\le
			\Pi\{\gamma\in\Gamma:e\in\gamma[\tau,\sigma]\}
			\int_\Gamma
			\frac{\1_{\{e\in\gamma[\tau,\sigma]\}}}{Z(\gamma)^2}\,\dd\Pi(\gamma)\\
			&\le \theta_e
			\int_\Gamma
			\frac{\1_{\{e\in\gamma[\tau,\sigma]\}}}{Z(\gamma)^2}\,\dd\Pi(\gamma).
		\end{align*}
		Since
		$\theta_e/b_e=w(\tail(e))-w(\head(e))$, Tonelli's theorem and telescoping
		along each selected segment yield
		\begin{align*}
			\cD_{\widehat D}(\Phi)
			&\le
			\int_\Gamma
			\frac{
				\sum_{e\text{ in the selected segment of }\gamma}
				\bigl(w(\tail(e))-w(\head(e))\bigr)}
			{Z(\gamma)^2}\,\dd\Pi(\gamma)\\
			&=\int_\Gamma
			\frac{w(x_{\tau(\gamma)})-w(x_{\sigma(\gamma)})}
			{Z(\gamma)^2}\,\dd\Pi(\gamma)\\
			&\le\int_\Gamma\frac{\dd\Pi(\gamma)}{Z(\gamma)}
			=|\Phi|.
		\end{align*}
		Here we used $Z(\gamma)=w(x_{\tau(\gamma)})>0$ and
		$w(x_{\sigma(\gamma)})\ge0$.  By Thomson's principle
		\eqref{eq:thomson-finite-network},
		\begin{equation*}
			\frac{|\Phi|^2}{\capacity_{\widehat D}(A,K)}
			\le\cD_{\widehat D}(\Phi)
			\le|\Phi|.
		\end{equation*}
		Combining this with \eqref{eq:selected-segment-strength} proves
		\eqref{eq:selected-segment-capacity}.
	\end{proof}
	
	Lemma~\ref{lem:selected-segment-capacity} is the capacitary input for the
	following central finite-domain lemma.
	
	\begin{lemma}\label{lem:finite-domain-capacity}
		Let $D\subset V$ be finite and connected, let $o\in D$, and suppose that
		$B_R\subseteq D$ for some $R>0$.  Let $w:\widehat D\to[0,\infty)$ vanish on
		$\cT_D$, and let $\lambda\ge0$ be defined by
		\eqref{eq:current-with-loss}.  Set
		\begin{equation*}
			p:=\sum_{t\in\cT_D}\lambda(t),
		\end{equation*}
		and assume that $p>0$.  Then
		\begin{equation}\label{eq:finite-domain-capacity}
			\sum_{x\in D}w(x)m(x)
			\ge cp\int_0^R
			\frac{r\,\dd r}
			{\log\!\left(e+w(o)\capacity_D(B_r)/p\right)},
		\end{equation}
		where $c>0$ is an absolute constant.
	\end{lemma}
	
	\begin{proof}
		Let
		\begin{equation*}
			\Gamma_\partial:=\{\gamma:\gamma\text{ ends in }\cT_D\}.
		\end{equation*}
		By \eqref{eq:flow-terminal-marginal}, this event has mass $p$.  Every terminal
		represents a vertex outside $D$, and hence outside $B_R$.  Therefore every path
		in $\Gamma_\partial$ reaches radial distance greater than $R$.
		
		Fix $0<r<R$.  On each path in $\Gamma_\partial$, select the segment from the
		tail of the first edge leaving $B_r$ to the terminal.  The segment may later
		re-enter $B_r$; this causes no difficulty because its divergence is supported
		only at its initial and terminal vertices.  Its initial voltage $Z_r$ is
		positive because voltage decreases strictly along every edge of the path.
		Since every initial vertex lies in $B_r$,
		Lemma~\ref{lem:selected-segment-capacity} applies with $A=B_r$ and
		$K=\cT_D$.  Together with \eqref{eq:ambient-terminal-capacity}, this gives
		\begin{equation*}
			\E_\Pi\left[Z_r^{-1}\,\middle|\,\Gamma_\partial\right]
			\le\frac{\capacity_{\widehat D}(B_r,\cT_D)}p
			=\frac{\capacity_D(B_r)}p.
		\end{equation*}
		
		By Lemma~\ref{lem:average-path-energy}, Lemma~\ref{lem:discrete-first-exit}
		applied with $R$, and Tonelli's theorem,
		\begin{equation*}
			\sum_{x\in D}w(x)m(x)
			\ge cp\int_0^R r\,
			\E_\Pi\left[
			\frac1{\log(e+w(o)/Z_r)}\,\middle|\,\Gamma_\partial
			\right]\dd r.
		\end{equation*}
		Since $p>0$, the event $\Gamma_\partial$ is nonempty.  Every path in it
		contains a directed edge, and hence $w(o)>0$.  The function
		\[
		t\longmapsto\frac1{\log(e+w(o)t)}
		\]
		is decreasing and convex on $[0,\infty)$.  Jensen's inequality, applied under
		the conditional probability
		$\Pi(\,\cdot\,\mid\Gamma_\partial)$, therefore gives
		\begin{align*}
			\E_\Pi\left[
			\frac1{\log(e+w(o)/Z_r)}\,\middle|\,\Gamma_\partial
			\right]
			&\ge
			\frac1{\log\!\left(
				e+w(o)\E_\Pi[Z_r^{-1}\mid\Gamma_\partial]
				\right)}\\
			&\ge
			\frac1{\log\!\left(
				e+w(o)\capacity_D(B_r)/p
				\right)}.
		\end{align*}
		Substitution into the preceding integral estimate proves
		\eqref{eq:finite-domain-capacity}.
	\end{proof}
	
	The relative-capacity form is an immediate consequence.
	
	\begin{corollary}\label{cor:finite-relative-capacity}
		Under the assumptions of Lemma~\ref{lem:finite-domain-capacity},
		\begin{equation*}
			\sum_{x\in D}w(x)m(x)
			\ge cp\int_0^{R/2}
			\frac{r\,\dd r}
			{\log\!\left(e+w(o)\capacity_{B_{2r}}(B_r)/p\right)}.
		\end{equation*}
	\end{corollary}
	
	\begin{proof}
		For $0<r<R/2$, one has $B_{2r}\subseteq D$, and domain monotonicity gives
		\[
		\capacity_D(B_r)\le\capacity_{B_{2r}}(B_r).
		\]
		Restricting \eqref{eq:finite-domain-capacity} to $(0,R/2)$ and using the fact
		that the integrand decreases with the capacity proves the claim.
	\end{proof}
	
	The following cutoff estimate is standard.  We include its short proof for
	completeness.
	
	\begin{lemma}\label{lem:intrinsic-cutoff}
		For every $r>0$,
		\begin{equation*}
			\capacity_{B_{2r}}(B_r)
			\le\frac{m(B_{2r})}{r^2}.
		\end{equation*}
	\end{lemma}
	
	\begin{proof}
		Use the radial cutoff
		\[
		\eta_r(x)=
		\begin{cases}
			1,&d_\rho(o,x)\le r,\\
			(2r-d_\rho(o,x))/r,&r<d_\rho(o,x)<2r,\\
			0,&d_\rho(o,x)\ge2r.
		\end{cases}
		\]
		Since $B_{2r}$ is finite, $\eta_r$ is finitely supported and admissible for
		$\capacity_{B_{2r}}(B_r)$.
		The path-metric property gives
		$|\eta_r(x)-\eta_r(y)|\le\rho(x,y)/r$ on every edge.  Only edges with at
		least one endpoint in $B_{2r}$ contribute.  Therefore
		\begin{align*}
			\capacity_{B_{2r}}(B_r)
			&\le\cE(\eta_r)\\
			&\le\frac1{r^2}
			\sum_{x\in B_{2r}}\sum_yb(x,y)\rho(x,y)^2\\
			&\le\frac{m(B_{2r})}{r^2}.
		\end{align*}
	\end{proof}
	
	\section{Capacity and stochastic completeness}\label{sec:killed}
	We now apply the finite-domain estimate to the killed $1$-resolvent.  This
	first gives the capacity-to-infinity criterion; the relative-capacity and
	volume criteria then follow by comparison and the intrinsic cutoff.
	Throughout this section, $\rho$ is an $m$-adapted edge-length function whose
	intrinsic path metric is complete, and $B_R=B_\rho(o,R)$.
	
	\subsection{Resolvent mass and boundary current}
	
	Let $D\subset V$ be finite and connected, with $o\in D$, and let $g_{1,D}$
	solve \eqref{eq:finite-one-resolvent}.  The following standard probabilistic
	representation follows from the killed Feynman--Kac and Laplace-transform
	formulas; see
	\cite[Section~2.1.1, p.~145, and Lemma~2.32, pp.~175--176]
	{KellerLenzWojciechowski2021}.
	
	\begin{lemma}\label{lem:resolvent-occupation}
		For every $x\in D$,
		\begin{equation}\label{eq:resolvent-occupation}
			g_{1,D}(x)m(x)
			=\E_o\int_0^{\tau_D}e^{-t}\1_{\{X_t=x\}}\,\dd t.
		\end{equation}
		Consequently, setting
		\begin{equation*}
			p_D:=\E_o e^{-\tau_D},
		\end{equation*}
		we have
		\begin{equation}\label{eq:resolvent-mass-split}
			\sum_{x\in D}g_{1,D}(x)m(x)=1-p_D.
		\end{equation}
	\end{lemma}

	The quantity $p_D$ is the Laplace transform of the exit time at parameter
	$1$.  It also equals the total current crossing the boundary.
	
	\begin{lemma}\label{lem:boundary-current}
		One has
		\begin{equation}\label{eq:boundary-current}
			\sum_{\substack{x\in D\\y\notin D}}b(x,y)g_{1,D}(x)=p_D.
		\end{equation}
	\end{lemma}
	
	\begin{proof}
		Sum \eqref{eq:network-one-resolvent} over $x\in D$.  Contributions of internal
		edges cancel, and the zero extension of $g_{1,D}$ gives
		\[
		\sum_{\substack{x\in D\\y\notin D}}b(x,y)g_{1,D}(x)
		+\sum_{x\in D}m(x)g_{1,D}(x)=1.
		\]
		Use \eqref{eq:resolvent-mass-split}.
	\end{proof}
	
	\begin{remark}
		Green-operator conservation identities for Dirichlet restrictions and killing
		terms are studied in a more general setting in
		\cite{HakeKellerPogorzelskiSchmidt2026}.  Here we use only the elementary
		finite-domain identity above.  The flow decomposition applied below is a
		separate ingredient.
	\end{remark}
	
	For $w=g_{1,D}$, definition \eqref{eq:current-with-loss} gives
	\begin{equation*}
		\lambda(x)=m(x)g_{1,D}(x),\quad x\in D,
		\qquad
		\lambda(t_e)=b(x,y)g_{1,D}(x)
	\end{equation*}
	for a boundary edge $e=\{x,y\}$, $x\in D$, $y\notin D$.  This
	formulation also includes absorption at the source, since
	$\lambda(o)=m(o)g_{1,D}(o)$ is allowed to be positive.
	
	\subsection{Capacity to infinity}
	
	By completeness and Lemma~\ref{lem:balls-connected}, $B_R$ is finite and
	connected.  For $D=B_R$, write
	\begin{equation*}
		g_{1,R}:=g_{1,B_R},
		\qquad
		p_R:=\E_oe^{-\tau_{B_R}}.
	\end{equation*}
	By \eqref{eq:resolvent-mass-split},
	\begin{equation*}
		\sum_{x\in B_R}g_{1,R}(x)m(x)=1-p_R.
	\end{equation*}
	By Lemma~\ref{lem:boundary-current}, the boundary mass $p$ in
	Lemma~\ref{lem:finite-domain-capacity} equals $p_R$.  Since $B_R$ is a proper
	finite connected set and $g_{1,R}>0$ on $B_R$,
	\eqref{eq:boundary-current} gives $p_R>0$, while
	\eqref{eq:resolvent-mass-split} gives $p_R<1$.  Applying the lemma with
	$D=B_R$ therefore gives, for $R>1$,
	\begin{equation*}
		1-p_R
		\ge cp_R\int_1^R
		\frac{r\,\dd r}
		{\log\!\left(e+g_{1,R}(o)\capacity_{B_R}(B_r)/p_R\right)}.
	\end{equation*}
	Since
	\begin{equation*}
		g_{1,R}(o)\le\frac{1-p_R}{m(o)},
	\end{equation*}
	and $t\mapsto1/\log(e+t)$ is decreasing, we obtain
	\begin{equation}\label{eq:killed-capacity-normalized}
		1-p_R
		\ge cp_R\int_1^R
		\frac{r\,\dd r}
		{\log\!\left(
			e+\frac{1-p_R}{p_R}\frac{\capacity_{B_R}(B_r)}{m(o)}
			\right)}.
	\end{equation}
	
	We use the elementary fact that for every fixed $A>0$ there is $C_A$ such
	that
	\begin{equation}\label{eq:fixed-log-factor}
		\log(e+At)\le C_A\log(e+t),
		\qquad t\ge0.
	\end{equation}
	Hence a fixed multiplicative factor inside the logarithm does not affect the
	divergence of the integrals considered below.
	
	The following standard exhaustion fact is recorded in
	\cite[Lemma~2.32 and the proof of Theorem~2.31, pp.~174--176]
	{KellerLenzWojciechowski2021}.
	\begin{lemma}\label{lem:exit-times-lifetime}
		If $D_n$ is an increasing finite exhaustion of $V$, then
		\begin{equation*}
			\tau_{D_n}\uparrow\zeta
			\qquad\text{almost surely}.
		\end{equation*}
		Consequently,
		\begin{equation*}
			\E_xe^{-\tau_{D_n}}\downarrow\E_xe^{-\zeta}.
		\end{equation*}
		Here and below, $e^{-\infty}:=0$.
	\end{lemma}

	\begin{proof}[Proof of Theorem~\ref{thm:intro-capacity-infinity}]
		The exit times increase with $R$, so $p_R$ decreases to a limit
		$p_\infty\ge0$.  Suppose that $p_\infty>0$.  Since
		$p_\infty\le p_R<1$, we have $0<p_\infty<1$.  Fix $L>1$ and let
		$R_j\uparrow\infty$, with $R_j>L$.  For every fixed $r\in[1,L]$,
		Lemma~\ref{lem:capacity-exhaustion} gives
		\[
		\frac{1-p_{R_j}}{p_{R_j}}\capacity_{B_{R_j}}(B_r)
		\longrightarrow
		\frac{1-p_\infty}{p_\infty}\capacity_V(B_r).
		\]
		Fatou's lemma and \eqref{eq:killed-capacity-normalized} yield
		\begin{align*}
			&\int_1^L
			\frac{r\,\dd r}
			{\log\!\left(
				e+\dfrac{1-p_\infty}{p_\infty}
				\dfrac{\capacity_V(B_r)}{m(o)}
				\right)}\\
			&\quad\le
			\liminf_{j\to\infty}\int_1^L
			\frac{r\,\dd r}
			{\log\!\left(
				e+\dfrac{1-p_{R_j}}{p_{R_j}}
				\dfrac{\capacity_{B_{R_j}}(B_r)}{m(o)}
				\right)}
			\le\frac{1-p_\infty}{cp_\infty},
		\end{align*}
		where the last inequality uses $L<R_j$.  Since $0<p_\infty<1$, condition
		\eqref{eq:intro-capacity-infinity-condition} and
		\eqref{eq:fixed-log-factor} make the left-hand side diverge as
		$L\to\infty$.  This contradiction shows that $p_\infty=0$.
		
		Lemma~\ref{lem:exit-times-lifetime} now gives $\E_oe^{-\zeta}=0$.  Since
		$e^{-\zeta}>0$ on $\{\zeta<\infty\}$, we have
		$\Pp_o(\zeta=\infty)=1$.  By the probabilistic characterization of stochastic
		completeness and the some/all base-point equivalence on connected graphs
		\cite[Theorem~7.2(i), pp.~310--312, and Theorem~7.32, pp.~352--354]
		{KellerLenzWojciechowski2021}, the graph is
		stochastically complete.
	\end{proof}
	
	The exhaustion in this proof is essential: replacing $\capacity_{B_R}(B_r)$
	by the smaller $\capacity_V(B_r)$ directly in
	\eqref{eq:finite-domain-capacity} would enlarge its right-hand side.
	
	\begin{remark}\label{rem:killed-role}
		The finite killed $1$-resolvent is useful here for three related reasons.  Its
		gradient is a finite acyclic current to which flow decomposition applies.  Its
		total mass is $1-p_D$, while its boundary current is $p_D$, by
		\eqref{eq:resolvent-mass-split} and \eqref{eq:boundary-current}.  Finally,
		$p_D=\E_oe^{-\tau_D}$ converges along an exhaustion to the Laplace transform
		of the lifetime.  Thus the proof needs neither a whole-graph Green kernel nor
		pointwise convergence of $g_{1,D}$.
	\end{remark}
	
	\subsection{Relative capacity and volume}
	
	\begin{proof}[Proof of Theorem~\ref{thm:intro-capacity}]
		By \eqref{eq:intro-global-local-capacity}, condition
		\eqref{eq:intro-capacity-condition} implies
		\eqref{eq:intro-capacity-infinity-condition}.
		Theorem~\ref{thm:intro-capacity-infinity} applies.
	\end{proof}
	
	We now recover the known Grigor'yan type integral criterion.  The following
	short argument is the final step of the new proof by flow decomposition.
	
	\begin{proof}[Proof of Corollary~\ref{cor:intro-volume}]
		Lemma~\ref{lem:intrinsic-cutoff} gives
		\[
		\capacity_{B_{2r}}(B_r)\le\frac{m(B_{2r})}{r^2}.
		\]
		For $r\ge1$,
		\[
		\log\!\left(e+\frac{\capacity_{B_{2r}}(B_r)}{m(o)}\right)
		\le
		\log\!\left(e+\frac{m(B_{2r})}{m(o)}\right).
		\]
		After the substitution $s=2r$, condition
		\eqref{eq:intro-volume-condition} implies
		\eqref{eq:intro-capacity-condition}.  The conclusion follows from
		Theorem~\ref{thm:intro-capacity}.
	\end{proof}
	
	\begin{remark}\label{rem:basepoint-volume}
		Condition \eqref{eq:intro-volume-condition} is independent of the base point.
		Indeed, if $d_\rho(o,o')=a$, then for $r\ge a$,
		\[
		B_\rho(o',r)\subset B_\rho(o,r+a)\subset B_\rho(o,2r).
		\]
		A fixed change of scale and the fixed factor $m(o)/m(o')$ inside the logarithm
		do not affect divergence.  Interchanging $o$ and $o'$ gives the converse
		implication.
	\end{remark}

	\section{Weighted half-line examples}\label{sec:half-line}
	
	The final section uses only weighted half-lines.  This setting retains the
	exact birth--death test needed for comparison while keeping the notation
	one-dimensional.  We first record the finite killed-resolvent identity and
	the exact test, and then give two examples.  The first separates the
	capacitary criteria from the volume criterion.  The second shows that a
	logarithmic loss can still remain even for capacity to infinity.
	
	Let $V=\mathbb N_0$, with an edge only between $n$ and $n+1$.  We use $0$ as
	the base point and write $B_r=B_\rho(0,r)$.  Set
	\begin{equation*}
		b_n:=b(n,n+1)>0,
		\qquad m_n:=m(n)>0,
		\qquad M_n:=\sum_{j=0}^nm_j.
	\end{equation*}
	For $D_N:=\{0,\ldots,N\}$, let $g_{1,D_N}$ solve
	\eqref{eq:network-one-resolvent} with pole $0$, and set
	\begin{equation*}
		h_n:=g_{1,D_N}(n),
		\qquad h_{N+1}:=0,
		\qquad p_N:=\E_0e^{-\tau_{D_N}}.
	\end{equation*}
	Define
	\begin{equation*}
		J_n:=b_n(h_n-h_{n+1}),
		\qquad 0\le n\le N,
		\qquad J_{-1}:=1.
	\end{equation*}
	The resolvent equation gives
	\begin{equation*}
		J_{n-1}-J_n=m_nh_n,
		\qquad 0\le n\le N.
	\end{equation*}
	In particular, $J_n$ decreases, and Lemma~\ref{lem:boundary-current} gives
	$J_N=p_N$.
	
	\begin{proposition}\label{prop:finite-half-line-identity}
		For every $N$,
		\begin{equation}\label{eq:finite-half-line-identity}
			1-p_N
			=\sum_{n=0}^Nm_nh_n
			=\sum_{n=0}^N\frac{M_nJ_n}{b_n}.
		\end{equation}
		Consequently,
		\begin{equation}\label{eq:finite-half-line-bound}
			p_N\le
			\left(1+\sum_{n=0}^N\frac{M_n}{b_n}\right)^{-1}.
		\end{equation}
	\end{proposition}
	
	\begin{proof}
		Abel summation gives
		\begin{equation*}
			\sum_{n=0}^Nm_nh_n
			=\sum_{n=0}^NM_n(h_n-h_{n+1})
			=\sum_{n=0}^N\frac{M_nJ_n}{b_n}.
		\end{equation*}
		The first equality in \eqref{eq:finite-half-line-identity} is
		\eqref{eq:resolvent-mass-split}.  Since $J_n\ge J_N=p_N$,
		\begin{equation*}
			1-p_N\ge p_N\sum_{n=0}^N\frac{M_n}{b_n},
		\end{equation*}
		which is equivalent to \eqref{eq:finite-half-line-bound}.
	\end{proof}
	
	For the converse direction, the standard expected-hitting-time equation
	\cite[Section~3.3, Theorem~3.3.3, p.~113]{Norris1997}, with Norris's
	generator $Q$ equal to our $\Delta$, gives
	\begin{equation}\label{eq:half-line-mean-exit}
		\E_0\tau_{D_N}=\sum_{n=0}^N\frac{M_n}{b_n}.
	\end{equation}

	The preceding identities yield the standard birth--death criterion; see
	\cite[Theorem~9.25, pp.~406--408]{KellerLenzWojciechowski2021}.
	
	\begin{theorem}\label{thm:known-half-line}
		The weighted half-line is stochastically complete if and only if
		\begin{equation}\label{eq:exact-half-line-test}
			\sum_{n=0}^\infty\frac{M_n}{b_n}=\infty.
		\end{equation}
	\end{theorem}
	
	\begin{proof}
		If the series diverges, \eqref{eq:finite-half-line-bound} gives $p_N\to0$.
		Together with Lemma~\ref{lem:exit-times-lifetime}, this implies
		$\E_0e^{-\zeta}=0$ and hence stochastic completeness by
		\cite[Theorem~7.2, pp.~310--312, and Theorem~7.32, pp.~352--354]
		{KellerLenzWojciechowski2021}.  If the series
		converges, \eqref{eq:half-line-mean-exit},
		Lemma~\ref{lem:exit-times-lifetime}, and monotone convergence give
		\begin{equation*}
			\E_0\zeta=\sum_{n=0}^\infty\frac{M_n}{b_n}<\infty.
		\end{equation*}
		Thus $\zeta<\infty$ almost surely from $0$, and the half-line is
		stochastically incomplete.
	\end{proof}
	
	We will also use the elementary series law for capacity.  If $0\le n\le k$,
	then
	\begin{equation}\label{eq:half-line-series-capacities}
		\begin{aligned}
			\capacity_{D_k}(D_n)^{-1}
			&=\sum_{j=n}^k\frac1{b_j},\\
			\capacity_V(D_n)^{-1}
			&=\sum_{j=n}^\infty\frac1{b_j},
		\end{aligned}
	\end{equation}
	with the extended-value convention.  The first equality is the series law
	for the edges from $\{n,n+1\}$ through $\{k,k+1\}$; the last edge is included
	because admissible functions vanish outside $D_k$.  The second equality
	follows from the first and Lemma~\ref{lem:capacity-exhaustion}.
	
	\subsection{Separating the capacity and volume criteria}
	
	The relative-capacity condition is strictly weaker than the volume condition,
	even for a weighted half-line with monotone edge weights and unbounded
	Laplacian.  Assign the edge $\{n,n+1\}$ the weight
	\[
	b_n:=b(n,n+1)=n+1.
	\]
	Thus the edge weights are strictly increasing.
	
	Set
	\[
	\mathcal J:=\{3,6,9,\ldots\},
	\]
	and define the vertex measure by
	\begin{equation*}
		m(n):=
		\begin{cases}
			1,&n\in\mathcal J,\\[1mm]
			\exp\!\bigl((n+1)^3\bigr),&n\notin\mathcal J.
		\end{cases}
	\end{equation*}
	For each $j\in\mathcal J$, give both edges adjacent to $j$ the length
	\[
	\ell_{j-1}=\ell_j:=\frac1{2\sqrt{j+1}},
	\]
	and give every other edge length one.  Since the pairs
	$\{j-1,j\}$, $j\in\mathcal J$, are disjoint, these lengths are
	well-defined.
	
	At an exceptional vertex $j\in\mathcal J$,
	\begin{align*}
		\sum_y b(j,y)\rho(j,y)^2
		&=
		j\ell_{j-1}^2+(j+1)\ell_j^2\\
		&=
		\frac{2j+1}{4(j+1)}
		<1=m(j).
	\end{align*}
	At a nonexceptional vertex $n$,
	\[
	\sum_y b(n,y)\rho(n,y)^2
	\le b_{n-1}+b_n
	=2n+1
	\le \exp\!\bigl((n+1)^3\bigr)
	=m(n).
	\]
	The corresponding estimate at the root is immediate.  Thus $\rho$ is
	$m$-adapted.
	
	Set
	\[
	r_0:=0,
	\qquad
	r_n:=\sum_{k=0}^{n-1}\ell_k,
	\qquad n\ge1.
	\]
	Among every three consecutive edges, at least one has length one.  Hence
	\begin{equation}\label{eq:half-line-radius-growth}
		r_n\asymp n.
	\end{equation}
	In particular, the intrinsic path metric is complete.
	
	The total jump rate at an exceptional vertex $j\in\mathcal J$ is
	\[
	\frac{b_{j-1}+b_j}{m(j)}
	=2j+1.
	\]
	The jump rates are therefore unbounded, and hence the Laplacian is an
	unbounded operator on $\ell^2(V,m)$.
	
	Let
	\[
	D_n:=\{0,1,\ldots,n\},
	\qquad
	M_n:=m(D_n).
	\]
	For all sufficiently large $n$,
	\[
	e^{n^3}\le M_n\le(n+1)e^{(n+1)^3}.
	\]
	Indeed, if $n\notin\mathcal J$, then
	$M_n\ge m(n)=e^{(n+1)^3}$, while if $n\in\mathcal J$, then
	$M_n\ge m(n-1)=e^{n^3}$.  Consequently,
	\begin{equation}\label{eq:half-line-volume-log}
		\log\!\left(e+\frac{M_n}{m(0)}\right)\asymp n^3.
	\end{equation}
	
	For $r_n\le r<r_{n+1}$, one has $B_\rho(0,r)=D_n$.  By
	\eqref{eq:half-line-radius-growth} and
	\eqref{eq:half-line-volume-log},
	\begin{align*}
		\int_{r_n}^{r_{n+1}}
		\frac{r\,\dd r}
		{\log\!\left(e+m(B_\rho(0,r))/m(0)\right)}
		&\le
		\frac{r_{n+1}\ell_n}
		{\log\!\left(e+M_n/m(0)\right)}\\
		&\lesssim\frac1{n^2}.
	\end{align*}
	It follows that
	\begin{equation}\label{eq:half-line-volume-condition-fails}
		\int_1^\infty
		\frac{r\,\dd r}
		{\log\!\left(e+m(B_\rho(0,r))/m(0)\right)}
		<\infty.
	\end{equation}
	Thus the volume condition fails.
	
	We next verify the relative-capacity condition.  Let
	$n\equiv1\pmod3$.  Then $\ell_n=1$.  For
	$r_n\le r<r_{n+1}$, one has $B_r=D_n$, and the function $\1_{D_n}$ is
	admissible for $\capacity_{B_{2r}}(B_r)$.  Since only the edge
	$\{n,n+1\}$ crosses the boundary of $D_n$,
	\[
	\capacity_{B_{2r}}(B_r)
	\le
	\cE(\1_{D_n})
	=
	b_n
	=
	n+1.
	\]
	Using \eqref{eq:half-line-radius-growth},
	\begin{align*}
		\int_{r_n}^{r_{n+1}}
		\frac{r\,\dd r}
		{\log\!\left(
			e+\capacity_{B_{2r}}(B_r)/m(0)
			\right)}
		&\ge
		\frac{r_n}
		{\log\!\left(e+(n+1)/m(0)\right)}\\
		&\gtrsim
		\frac{n}{\log n}.
	\end{align*}
	Summing over $n\equiv1\pmod3$ gives
	\begin{equation}\label{eq:half-line-capacity-condition-holds}
		\int_1^\infty
		\frac{r\,\dd r}
		{\log\!\left(
			e+\capacity_{B_{2r}}(B_r)/m(0)
			\right)}
		=\infty.
	\end{equation}
	Hence the relative-capacity criterion applies although the volume criterion
	does not.
	
	The capacity-to-infinity criterion obviously applies as well.  Indeed,
	\eqref{eq:half-line-series-capacities} and $b_j=j+1$ give
	$\capacity_V(D_n)=0$.  Every intrinsic ball is some $D_n$, so the denominator
	in \eqref{eq:intro-capacity-infinity-condition} is $\log e=1$ and that
	integral diverges.
	
	For completeness, the exact half-line criterion gives the same conclusion.
	Indeed, $b_n=n+1$, while $M_n\ge e^{n^3}$ for all sufficiently large $n$.
	Thus
	\[
	\sum_{n=0}^\infty\frac{M_n}{b_n}=\infty,
	\]
	and Theorem~\ref{thm:known-half-line} shows that the half-line is
	stochastically complete.
	
	\subsection{A remaining logarithmic loss}
	
	Fix $\gamma\in\mathbb R$ and define integers
	\begin{equation*}
		\beta_n
		:=\max\left\{
		2,\left\lceil(n+2)(\log(n+2))^\gamma\right\rceil
		\right\}.
	\end{equation*}
	Starting from $a_0:=1$, set
	\begin{equation*}
		a_{n+1}:=\beta_na_n,
		\qquad
		m(n):=a_n,
		\qquad
		b_n:=b(n,n+1):=a_{n+1}.
	\end{equation*}
	Since $\beta_n\ge2$,
	\begin{equation*}
		a_n\le M_n=\sum_{j=0}^na_j\le2a_n.
	\end{equation*}
	Consequently,
	\begin{equation*}
		\frac{M_n}{b_n}
		\asymp\frac1{\beta_n}
		\asymp\frac1{n(\log n)^\gamma}.
	\end{equation*}
	The exact half-line criterion \eqref{eq:exact-half-line-test} therefore gives
	\begin{equation*}
		\text{stochastic completeness}
		\quad\Longleftrightarrow\quad
		\gamma\le1.
	\end{equation*}
	
	Give the edge $\{n,n+1\}$ the length
	\begin{equation*}
		\ell_n:=\frac1{\sqrt{2\beta_n}}.
	\end{equation*}
	At the root,
	\begin{equation*}
		b_0\ell_0^2=\frac{a_0}{2}\le m(0),
	\end{equation*}
	while, for $n\ge1$,
	\begin{align*}
		b_{n-1}\ell_{n-1}^2+b_n\ell_n^2
		&=\frac{a_n}{2\beta_{n-1}}+\frac{a_{n+1}}{2\beta_n}\\
		&=\frac{a_n}{2\beta_{n-1}}+\frac{a_n}{2}
		\le a_n=m(n).
	\end{align*}
	Thus $\rho$ is $m$-adapted.  If
	\begin{equation*}
		r_0:=0,
		\qquad r_n:=\sum_{j=0}^{n-1}\ell_j,
	\end{equation*}
	then standard summation and integral comparison give
	\begin{equation}\label{eq:log-loss-asymptotics}
		r_n\asymp\sqrt n\,(\log n)^{-\gamma/2},
		\qquad
		\log M_n\asymp n\log n.
	\end{equation}
	In particular, $r_n\to\infty$, so the intrinsic path metric is complete.  The
	jump rate at $n\ge1$ is $1+\beta_n$, and hence the Laplacian is unbounded.
	
	For $r_n\le r<r_{n+1}$, one has $B_\rho(0,r)=D_n$.  It follows from
	\eqref{eq:log-loss-asymptotics} that
	\begin{align}
		\int_{r_n}^{r_{n+1}}
		\frac{r\,\dd r}{\log(e+M_n)}
		&=\frac{r_n\ell_n+\ell_n^2/2}{\log(e+M_n)} \notag\\
		&\asymp\frac1{n(\log n)^{\gamma+1}}.
		\label{eq:half-line-log-loss-increment}
	\end{align}
	Therefore the volume condition holds precisely when $\gamma\le0$.
	
	The relative-capacity condition has the same threshold.  For
	$r_n\le r<r_{n+1}$, let
	\begin{equation*}
		k=k(r):=\max\{j:r_j\le2r\},
	\end{equation*}
	so that $B_{2r}=D_k$.  The first identity in
	\eqref{eq:half-line-series-capacities} gives
	\begin{equation*}
		\capacity_{B_{2r}}(B_r)^{-1}
		=\sum_{j=n}^k\frac1{b_j}.
	\end{equation*}
	Since $b_{j+1}/b_j=\beta_{j+1}\ge2$,
	\begin{equation*}
		\frac{b_n}{2}
		\le\capacity_{B_{2r}}(B_r)
		\le b_n.
	\end{equation*}
	Together with $\log b_n\asymp n\log n$, this shows that the relative-capacity
	integral has the same increments as
	\eqref{eq:half-line-log-loss-increment}.  It therefore also detects precisely
	the range $\gamma\le0$.
	
	Capacity to infinity has the same order.  The second identity in
	\eqref{eq:half-line-series-capacities} gives
	\[
	\capacity_V(D_n)^{-1}=\sum_{j=n}^\infty\frac1{b_j}.
	\]
	Since $b_{j+1}/b_j\ge2$, one again has
	\[
	\frac{b_n}{2}\le\capacity_V(D_n)\le b_n.
	\]
	Thus the capacity-to-infinity criterion also has the threshold $\gamma\le0$.
	
	Thus the capacity-to-infinity, relative-capacity, and volume criteria all miss
	the complete range
	\begin{equation*}
		0<\gamma\le1.
	\end{equation*}
	The loss occurs before the intrinsic cutoff \eqref{eq:intro-cutoff}: it is
	already present in both logarithmic capacitary estimates.  In contrast, the
	exact birth--death test retains the individual terms $M_n/b_n$.
	
	\begin{remark}
		This half-line is the radial quotient of the rooted spherically symmetric tree
		in which every vertex at level $n$ has $\beta_n$ children, with unit edge
		weights and unit vertex measure.  Indeed, the $n$-th sphere has $a_n$ vertices,
		while the cut leading to the next sphere has total conductance $a_{n+1}$.
		This observation motivates the example but is not used in its proof.
	\end{remark}
	
	Together, the two half-line examples show that the logarithmic capacity
	conditions can hold when the volume condition fails, but neither is sharp on
	the second half-line.  The exact birth--death current identity retains more
	information than either capacitary estimate. The second example therefore points to the need for a further refinement if
	one seeks a criterion that is sharp on all half-line examples or equivalently, weakly spherically symmetric graphs.

\end{document}